\documentclass[12pt,a4paper]{article}%
\usepackage[utf8]{inputenc}
\usepackage{hyperref}
\usepackage{amsmath}
\usepackage{amsfonts}
\usepackage{amssymb}
\usepackage{graphicx}
\usepackage[normalem]{ulem}
\usepackage{color}%
\providecommand{\U}[1]{\protect\rule{.1in}{.1in}}
\newtheorem{theorem}{Theorem}

\newtheorem{corollary}[theorem]{Corollary}

\newtheorem{lemma}[theorem]{Lemma}

\newtheorem{problem}[theorem]{Problem}
\newtheorem{proposition}[theorem]{Proposition}

\newtheorem{observation}[theorem]{Observation}

\newenvironment{proof}[1][Proof]{\noindent\textbf{#1.} }{\ \hfill \rule{0.5em}{0.5em}\bigskip}
\graphicspath{{D:/Dropbox/Riste-Sedlar/ModularP/Slike/}}

\begin{document}

\title{Domination number of modular product graphs}
\author{Sergio Bermudo$^{a}$, Iztok Peterin$^{b,c}$, Jelena Sedlar$^{d,e}$, Riste
\v{S}krekovski$^{e,f}$\\{\small $^{a}$ \textit{Department of Economics, Quantitative Methods and
Economic History, }}\\{\small \textit{Pablo de Olavide University, Carretera de Utrera Km. 1,
ES-41013 Seville, Spain}}\\[0.1cm] {\small $^{b}$ \textit{University of Maribor, FEECS, Maribor, Slovenia
}}\\[0.1cm] {\small $^{c}$ \textit{Institute of Mathematics, Physics and
Mechanics, Ljubljana, Slovenia }}\\[0.1cm] {\small $^{d}$ \textit{University of Split, Faculty of civil
engineering, architecture and geodesy, Split, Croatia }}\\[0.1cm] {\small $^{e}$ \textit{Faculty of Information Studies, Novo Mesto,
Slovenia}}\\[0.1cm] {\small $^{f}$ \textit{University of Ljubljana, FMF, Ljubljana,
Slovenia }}}
\maketitle

\begin{abstract}
The modular product $G\diamond H$ of graphs $G$ and $H$ is a graph on vertex
set $V(G)\times V(H)$. Two vertices $(g,h)$ and $(g^{\prime},h^{\prime})$ of
$G\diamond H$ are adjacent if $g=g^{\prime}$ and $hh^{\prime}\in E(H)$, or
$gg^{\prime}\in E(G)$ and $h=h^{\prime}$, or $gg^{\prime}\in E(G)$ and
$hh^{\prime}\in E(H)$, or (for $g\neq g^{\prime}$ and $h\neq h^{\prime}$)
$gg^{\prime}\notin E(G)$ and $hh^{\prime}\notin E(H)$. A set $D\subseteq V(G)$
is a dominating set of $G$ if every vertex outside of $D$ contains a neighbor
in $D$. A set $D\subseteq V(G)$ is a total dominating set of $G$ if every
vertex of $G$ contains a neighbor in $D$. The domination number $\gamma(G)$
(resp. total domination number $\gamma_{t}(G)$) of $G$ is the minimum
cardinality of a dominating set (resp. total dominating set) of $G$. In this
work we give several upper and lower bounds for $\gamma(G\diamond H)$ in terms
of $\gamma(G),$ $\gamma(H)$, $\gamma_{t}(\overline{G})$ and $\gamma
_{t}(\overline{H})$, where $\overline{G}$ is the complement graph of $G$.
Further, we fully describe graphs where $\gamma(G\diamond H)=k$ for
$k\in\{1,2,3\}$. Several conditions on $G$ and $H$ under which $\gamma
(G\diamond H)$ is at most $4$ and $5$ are also given. A new type of
simultaneous domination $\bar{\gamma}(G)$, defined as the smallest number of
vertices that dominates $G$ and totally dominates the complement of $G,$
emerged as useful and we believe it could be of independent interest. We
conclude the paper by proposing few directions for possible further research.

\end{abstract}



\section{Introduction}

Any product of two graphs $G$ and $H$ has a vertex set $V(G)\times V(H)$, and
the set of edges of a graph product can be defined in many different ways. One
can consider three objects for the definition of an edge in a product:
consider either a vertex or an edge or a non-edge in one factor and combine it
with the same objects in the other factor. This yields eight different
possibilities (notice that vertex by vertex must be ignored to avoid loops) to
have an edge or a non-edge in a product, which gives $2^{8}=256$ different
graph products, see \cite{HaIK, Imri1}.

Some of them are rather ``uninteresting'', but over the years four of
them---Cartesian, direct, strong and lexicographic---gain a special status and
we call them standard products. All four of them are associative, but the
lexicographic is not commutative. All in all, there are exactly ten products
that are both associative and commutative and we can join them into pairs
together with their complementary products. Six of them represent either a
standard product or its complementary product, two of them are trivial product
(no edges) and complete product (all edges) and the remaining two products are
modular product and its complementary product.

Modular product did not gain too much attention over the years from the
community. Seminal paper on modular product is due to Imrich \cite{Imri1} and
contains a discussion on algebraic properties. Later one can find a work on
$L(2,1)$-labelings of modular products by Shao and Solis-Oba \cite{shao-solis}%
. Another work on modular product brings an upper bound on broadcast
dominating number by Sen and Kola \cite{SeKo}. A recent manuscript \cite{KKPY}
by Kang et. al. contains a distance formula and a method to obtain a strong
metric dimension of a modular product.

It is convenient to warn that there exits a confusion with the name modular
product. While we follow the book \cite{HaIK}, several authors use the name
modular product for a product that is called the weak modular product in
\cite{HaIK} and the direct-co-direct product in \cite{KePe,KePe1}. In any case
we recommend caution.

Domination problems have a long history with respect to graph products,
starting with the Cartesian product and the long standing Vizing conjecture,
which claims that $\gamma(G\Box H)\geq\gamma(G)\gamma(H)$ for any two graphs
$G$ and $H$. Despite several attempts and different approaches to this
conjecture, it is still open. The last survey about this conjecture was
written by Bre\v{s}ar et. al. \cite{BrDoKl} and since then one can find some
other attempts to it, see \cite{Bres1,Bres2,BHHK,PiPS,Zerb}, by different authors.

Among other products, the direct product gained the biggest attention with
respect to the domination number, see Bre\v{s}ar et. al. \cite{BrKR} for an
upper bound and Meki\v{s} \cite{Meki} for a lower bound on the domination
number of the direct product. Nowakowski and Rall \cite{NoRa} used another
approach to several associative graph products, they studied several of them
with respect to some chromatic, independent and dominating parameters. In
particular, a subsection of \cite{NoRa} is devoted to the domination number of
modular product, which they call equivalence product.

In the present work we consider $\gamma(G\diamond H)$ from several
perspectives. In the next section we first settle the notation. This is
followed by a section that relates $\gamma(G\diamond H)$ to some other graph
parameters and yields several exact values of $\gamma(G\diamond H)$, that will
be useful for later results. In the fourth section we characterize graphs with
$\gamma(G\diamond H)=k$ for $k\in\{1,2,3\}$. We finish the work with some
conclusions and open problems.

\section{Preliminaries}

We consider only simple, undirected and finite graphs. For a graph $G$, we
denote by $\overline{G}$ the \emph{complement} of $G$, that is, the graph on
the same vertex set $V(G)$ such that, for any two different vertices $u$ and
$v$ of $G,$ $uv\in E(\overline{G})$ if and only if $uv\notin E(G)$. The
\emph{distance} between two vertices $u,v\in V(G)$ is the minimum number of
edges on any path in $G$ that starts in $u$ and ends in $v$, and is denoted by
$d_{G}(u,v)$. If such a path does not exist, then we have $d_{G}(u,v)=\infty$.
The maximum distance among any pair of vertices of $G$ is the \emph{diameter}
of $G$, and it is denoted by $\mathrm{diam}(G)$.

As usual, $N_{G}(v)$ and $N_{G}[v]$ denote the open and the closed
neighborhood, respectively, of a vertex $v$ of $G$, it means, $N_{G}(v)=\{u\in
V(G):uv\in E(G)\}$ and $N_{G}[v]=N_{G}(v)\cup\{v\}$. Vertex $v$ is called a
\emph{universal vertex} when $N_{G}[v]=V(G)$ and \emph{isolated vertex} when
$N_{G}[v]=\{v\}$. Given a set of vertices $D\subseteq V(G)$ and a vertex $v\in
D$, the \emph{open neighborhood} of $D$ is $N_{G}(D)=\bigcup_{u\in D}N_{G}%
(u)$, the \emph{closed neighborhood} of $D$ is $N_{G}[D]=N_{G}(D)\cup D$, and
the \emph{private neighbors} of $v\in D$ are the vertices in the set
$\mathrm{pr}[v,D]=N_{G}[v]\setminus N_{G}[D\backslash\{v\}]$. A vertex $u$ is
said to be \emph{dominated} by a vertex $v$ if $u\in N[v].$

A set $D\subseteq V(G)$ is a \emph{dominating set} of $G$ if $N_{G}[D]=V(G)$.
The \emph{domination number} $\gamma(G)$ of $G$ is the minimum cardinality of
a dominating set of $G$. Every dominating set of cardinality $\gamma(G)$ is
called a $\gamma(G)$\emph{-set}. If further $\{N_{G}[v]:v\in D\}$ is a
partition of $V(G)$, then $G$ is an \emph{efficiently closed dominated} graph,
or an ECD graph for short, and $D$ is an \emph{efficiently closed dominating}
set of $G$, or an ECD set for short. ECD graphs have been largely studied over
the decades, also from the perspective of $1$-perfect codes, see \cite{KPY}
and the references there in. A set $D\subseteq V(G)$ is a \emph{total
dominating set} of a graph $G$ if $N_{G}(D)=V(G)$. The \emph{total domination
number} $\gamma_{t}(G)$ of $G$ is the minimum cardinality of a total
dominating set of $G$ or infinite if such a set does not exist. Every total
dominating set of cardinality $\gamma_{t}(G)$ is called a $\gamma_{t}%
(G)$\emph{-set}. The \emph{packing number} $\rho(G)$ of a graph $G$ is the
maximum number of vertices of $G$ such that their closed neighborhoods are
pairwise disjoint.

We will use a new concept for a graph $G$ and its complement $\overline{G}$.
Let $D$ be at the same time a dominating set of $G$ and a total dominating set
of $\overline{G}$, then we say that $D$ is a \emph{simultaneously dominating
and complement total dominating set} (SDCTD set for short) of $G$. The minimum
cardinality of SDCTD set of $G$ is denoted by $\bar{\gamma}(G)$ and is called
the SDCTD number of $G$. Clearly, $\bar{\gamma}(G)$ exists if there are no
isolated vertices in $\overline{G}$, which means no universal vertices in $G$.
To the best of our knowledge, this parameter has not been studied yet.

Given two factors $G$ and $H$, a pair of vertices $(g,h)$ and $(g^{\prime
},h^{\prime})$ in $V(G)\times V(H)$ forms an edge of the

\begin{itemize}
\item \emph{Cartesian product} $G\Box H$, if $g=g^{\prime}$ and $hh\in E(H)$,
or $gg^{\prime}\in E(G)$ and $h=h^{\prime}$;

\item \emph{direct product} $G\times H$, if $gg^{\prime}\in E(G)$ and
$hh^{\prime}\in E(H)$.
\end{itemize}

\noindent Now, we define the edges of the \emph{modular product} $G\diamond H$
as
\[
E(G\diamond H)=E(G\Box H)\cup E(G\times H)\cup E(\overline{G}\times
\overline{H}).
\]
The three sets that define $E(G\diamond H)$ are pairwise disjoint. Therefore,
we call \emph{Cartesian edges} to the edges of $E(G\diamond H)$ that belong to
the Cartesian product, \emph{direct edges} to those which belong to $G\times
H$, and \emph{co-direct edges} to the ones in the direct product $\overline
{G}\times\overline{H}$ of complements of $G$ and $H$. Recall that the edge set
of the \emph{strong product} $G\boxtimes H$ consists of the Cartesian edges
and the direct edges between $G$ and $H$. Besides, let us make the following observation.

\begin{observation}
\label{Obs_adjacency}A vertex $(g_{1},h_{1})$ is dominated by $(g_{2},h_{2})$
in $G\diamond H$ if and only if either $g_{1}\in N_{G}[g_{2}]$ and $h_{1}\in
N_{H}[h_{2}]$ or $g_{1}\notin N_{G}[g_{2}]$ and $h_{1}\notin N_{H}[h_{2}]$.
\end{observation}

The last of the standard products is the \emph{lexicographic product} $G\circ
H$, where $(g,h)(g^{\prime},h^{\prime})$ is an edge if $gg^{\prime}\in E(G)$
or ($g=g^{\prime}$ and $hh^{\prime}\in E(H)$). Clearly, lexicographic product
is not commutative because its edge set is not defined symmetrically.
Different products of particular graphs can be isomorphic, the following
simple connection is important in this work:
\begin{equation}
\label{complete}G\boxtimes K_{t}\cong G\circ K_{t}\cong G\diamond K_{t},
\end{equation}
where $K_{t}$ is the complete graph with $t$ vertices.


\section{Some bounds and exact values}

\label{sec1}

In this section we give some bounds and exact values for the domination number
of a modular product.

\paragraph{Lower bound.}

We start with a lower bound of $\gamma(G\diamond H)$ using the domination
number in one factor and the total domination number in the complement of the
other factor. This is not a big surprise due to the co-direct edges of modular
product. For this we need the following independent result, where by $[k]$ we
denote the set $\{1,\ldots,k\}$ here and throughout the paper.

\begin{proposition}
\label{nodom} Let $G$ and $H$ be two graphs. A set $D=\{(g_{1},h_{1}%
),\ldots,(g_{k},h_{k})\}$ is a dominating set in $G\diamond H$ if and only if
for every $I\subseteq\lbrack k]$ there is no vertex $(g,h)$ in $G\diamond H$
such that $N_{G}[g]\cap\{g_{1},\ldots,g_{k}\}=\{g_{i}:i\in I\}$ and
$N_{H}[h]\cap\{h_{1},\ldots,h_{k}\}=\{h_{i}:i\in\lbrack k]\setminus I\}$.
\end{proposition}

\begin{proof}
The set $D$ is a dominating set in $G\diamond H$ if and only if for every
$(g,h)$ of $G\diamond H$ there exists $i\in\lbrack k]$ such that either
$(g,h)=(g_{i},h_{i})$ or $(g,h)$ is adjacent to $(g_{i},h_{i}).$ According to
Observation \ref{Obs_adjacency}, this is equivalent to either $g_{i}\in
N_{G}[g]$ and $h_{i}\in N_{H}[h]$ or $g_{i}\not \in N_{G}[g]$ and
$h_{i}\not \in N_{H}[h].$ Let us consider any set $I\subseteq\lbrack k]$ and
any vertex $(g,h).$ Let $i\in\lbrack k]$ be such an index that $g_{i}\in
N_{G}[g]$ and $h_{i}\in N_{H}[h]$ or $g_{i}\not \in N_{G}[g]$ and
$h_{i}\not \in N_{H}[h].$ 

Assume first that $g_{i}\in N_{G}[g]$ and $h_{i}\in N_{H}[h].$ Notice that
either $i\in I$ or $i\in\lbrack k]\setminus I.$ If $i\in I,$ then $h_{i}\in
N_{H}[h]$ implies $N_{H}[h]\cap\{h_{1},\ldots,h_{k}\}\not =\{h_{i}:i\in\lbrack
k]\setminus I\}.$ On the other hand, if $i\in\lbrack k]\setminus I,$ then
$g_{i}\in N_{G}[g]$ implies $N_{G}[g]\cap\{g_{1},\ldots,g_{k}\}\not =%
\{g_{i}:i\in I\}.$

Assume next that $g_{i}\not \in N_{G}[g]$ and $h_{i}\not \in N_{H}[h].$ If
$i\in I,$ then $g_{i}\not \in N_{G}[g]$ implies $N_{G}[g]\cap\{g_{1}%
,\ldots,g_{k}\}\not =\{g_{i}:i\in I\}.$ And if $i\in\lbrack k]\setminus I,$
then $h_{i}\not \in N_{H}[h]$ implies $N_{H}[h]\cap\{h_{1},\ldots
,h_{k}\}\not =\{h_{i}:i\in\lbrack k]\setminus I\}.$

Either way, we have obtained that for a set $I\subseteq\lbrack k]$ there is no
vertex $(g,h)$ in $G\diamond H$ such that $N_{G}[g]\cap\{g_{1},\ldots
,g_{k}\}=\{g_{i}:i\in I\}$ and $N_{H}[h]\cap\{h_{1},\ldots,h_{k}%
\}=\{h_{i}:i\in\lbrack k]\setminus I\}$.

\smallskip

Let us now prove the other direction of the equivalence. Assume that for every
$I\subseteq\lbrack k]$ there is no vertex $(g,h)$ in $G\diamond H$ such that
$N_{G}[g]\cap\{g_{1},\ldots,g_{k}\}=\{g_{i}:i\in I\}$ and $N_{H}[h]\cap
\{h_{1},\ldots,h_{k}\}=\{h_{i}:i\in\lbrack k]\setminus I\}$. This implies that
for every vertex $(g,h)$ of $G\diamond H$ there exists $i\in\lbrack k]$ such
that either $g_{i}\in N_{G}[g]$ and $h_{i}\in N_{H}[h]$ or $g_{i}%
\not \in N_{G}[g]$ and $h_{i}\not \in N_{H}[h],$ and we are done.
\end{proof}

If we take $I=\emptyset$ in the above proposition, we get the following
interesting corollary.

\begin{corollary}
\label{Cor_sufficient}Let $G$ and $H$ be two graphs. If $D=\{(g_{1}%
,h_{1}),\ldots,(g_{k},h_{k})\}$ is a dominating set in $G\diamond H$, then
$\{g_{1},\ldots,g_{k}\}$ is a dominating set in $G$ or $\{h_{1},\ldots
,h_{k}\}$ is a total dominating set in $\overline{H}$.
\end{corollary}

Let us now consider a $\gamma(G\diamond H)$-set $D=\{(g_{1},h_{1}%
),\ldots,(g_{k},h_{k})\}.$ If $\gamma(G)>\gamma(G\diamond H)$, then
$\{g_{1},\ldots,g_{k}\}$ is not a dominating set in $G$, so Corollary
\ref{Cor_sufficient} gives $\gamma_{t}(\overline{H})\leq\gamma(G\diamond H)$.
The similar argument can be applied to $\gamma(H)$ and $\gamma_{t}%
(\overline{G}),$ thus the above auxiliary results yield the following theorem.

\begin{theorem}
\label{Lobound} For any two graphs $G$ and $H$ we have
\[
\max\{\min\{\gamma(G),\gamma_{t}(\overline{H})\},\min\{\gamma(H),\gamma
_{t}(\overline{G})\}\}\leq\gamma(G\diamond H).
\]

\end{theorem}

The mentioned lower bound is sharp for an arbitrary graph $G$ and $H=K_{n}$,
as we will see later in Corollary \ref{Kn}.

The following simple result, which we are not aware of its existence in the
literature, allows us to simplify the lower bound for $\gamma(G\diamond H)$
given in Theorem \ref{Lobound} even more.

\begin{proposition}
\label{total2} For any graph $H$, $\mathrm{diam}(H)\geq3$ if and only if
$\gamma_{t}(\overline{H})=2$.
\end{proposition}

\begin{proof}
On one hand, any set $\{h_{1},h_{2}\}$ with $d_{H}(h_{1},h_{2})\geq3$ is a
total dominating set in $\overline{H}$. On the other hand, if $\{h_{1}%
,h_{2}\}$ is a total dominating set in $\overline{H}$, then $h_{1}$ is not
adjacent to $h_{2}$ in $H$ and there exists no vertex $h_{3}\in V(H)$ such
that both $h_{1}$ and $h_{2}$ are adjacent to $h_{3}$ in $H$. Therefore,
$d_{H}(h_{1},h_{2})\geq3$.
\end{proof}

Plugging the result of Proposition \ref{total2} into the bound of Theorem
\ref{Lobound} immediately yields the following result.

\begin{corollary}
\label{cor1}For any two graphs $G$ and $H$, $\mathrm{diam}(H)=2$ implies
\[
\min\{\gamma(G),3\}\leq\gamma(G\diamond H).
\]

\end{corollary}

\begin{figure}[h]
\centering\includegraphics[scale=0.8]{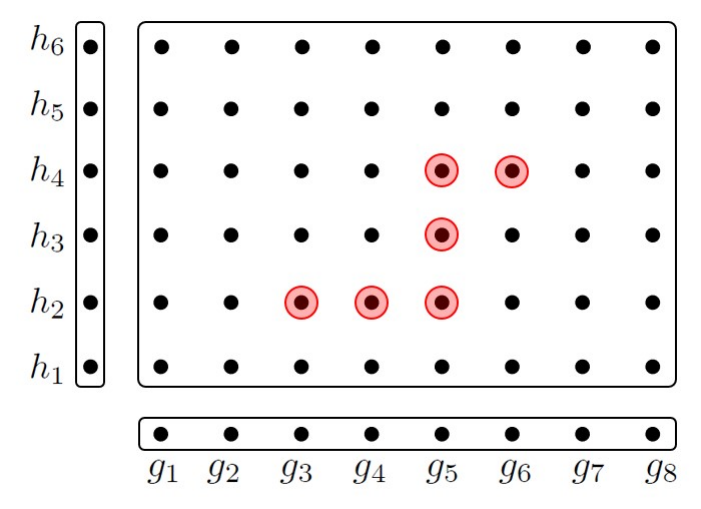}\caption{The figure shows
vertices of graphs $G,$ $H$ and $G\diamond H.$ The set $D$ of vertices in
$G\diamond H$ is highlighted and it forms a snake. The vertex $(g_{5},h_{2})$
is a $H$-corner and the vertex $(g_{5},h_{4})$ is a $G\,$-corner in $D.$}%
\label{Fig04}%
\end{figure}

\paragraph{Upper bound.}

We switch now to the upper bound, for which we need several notions. Let
$D=\{(x_{i},y_{i}):i\in\lbrack q]\}$ be an ordered set of vertices in
$G\diamond H.$ A set $D$ is a \emph{snake} in $G\diamond H$ if for every
$i\in\lbrack q-1]$ it holds that either $x_{i}=x_{i+1}$ or $y_{i}=y_{i+1}.$
Notice that for $q=1$ a set $D$ is trivially a snake. An example of a snake is
shown in Figure \ref{Fig04}. A \emph{projection} of a set $D$ onto $G$ (resp.
$H$) is the set defined by $\mathrm{proj}_{G}(D)=\{x_{i}:i\in\lbrack q]\}$
(resp. $\mathrm{proj}_{H}(D)=\{y_{i}:i\in\lbrack q]\}$). For the snake
$D=\{(g_{3},h_{2}),(g_{4},h_{2}),(g_{5},h_{2}),(g_{5},h_{3}),(g_{5}%
,h_{4}),(g_{6},h_{4})\}$ from Figure \ref{Fig04}, it holds that $\mathrm{proj}%
_{G}(D)=\{g_{3},g_{4},g_{5}\}$ and $\mathrm{proj}_{H}(D)=\{h_{2},h_{3}%
,h_{4}\}.$ Besides, for any claim on $G$ and $H$, its \emph{mirror} claim will
be the same claim when we have exchanged the roles of $G$ and $H$.

\begin{theorem}
\label{Tm79}Let $D=\{(x_{i},y_{i}):i\in\lbrack q]\}$ be a snake in $G\diamond
H$ with one of the following conditions

\begin{itemize}
\item[\emph{(i)}] $\mathrm{proj}_{G}(D)$ is a dominating set in $G$ and
$\mathrm{proj}_{H}(D)$ is a dominating set in $H$;

\item[\emph{(ii)}] $\mathrm{proj}_{G}(D)$ is a total dominating set in
$\overline{G}$ and $\mathrm{proj}_{H}(D)$ is a total dominating set in
$\overline{H}$.
\end{itemize}

\noindent Then, $D$ is a dominating set in $G\diamond H$.
\end{theorem}

\begin{proof}
Assume to the contrary, that $D$ is not a dominating set in $G\diamond H$.
Proposition \ref{nodom} implies that there must exist $I\subseteq\lbrack q]$
and a vertex $(g,h)$ in $G\diamond H$ such that $N[g]\cap\mathrm{proj}%
_{G}(D)=\{x_{i}:i\in I\}$ and $N[h]\cap\mathrm{proj}_{H}(D)=\{y_{i}%
:i\in\lbrack q]\setminus I\}$. Let us show that for every $i\in I,$ the index
$i+1$ must also belong to $I.$ Notice that by the definition of $D$ it holds
that either $x_{i}=x_{i+1}$ or $y_{i}=y_{i+1}.$ If $x_{i}=x_{i+1},$ then
$x_{i}\in N[g]\cap\mathrm{proj}_{G}(D)$ implies $x_{i+1}\in N[g]\cap
\mathrm{proj}_{G}(D),$ so $i+1\in I.$ If $y_{i}=y_{i+1},$ then $y_{i}%
\not \in N[h]\cap\mathrm{proj}_{H}(D)$ implies $y_{i+1}\not \in
N[h]\cap\mathrm{proj}_{H}(D),$ so again $i+1\in I.$ Hence, we have established
that $i\in I$ implies $i+1\in I.$ Similarly, $i\in I$ implies $i-1\in I.$ We
conclude that either $I=\emptyset$ or $I=[q],$ which implies either
$g\not \in N_{G}[\mathrm{proj}_{G}(D)]$ or $h\not \in N_{H}[\mathrm{proj}%
_{H}(D)],$ so either $\mathrm{proj}_{G}(D)$ is not a dominating set in $G$ or
$\mathrm{proj}_{H}(D)$ is not a dominating set in $H,$ a contradiction with
$(i)$. Further, $I=\emptyset$ or $I=[q]$ also implies that $N[h]\cap
\mathrm{proj}_{H}(D)=\{y_{i}:i\in\lbrack q]\}$ or $N[g]\cap\mathrm{proj}%
_{G}(D)=\{x_{i}:i\in\lbrack q]\},$ which means $\mathrm{proj}_{H}(D)$ is not a
$\gamma_{t}(\overline{H})$-set or $\mathrm{proj}_{G}(D)$ is not a $\gamma
_{t}(\overline{G})$-set, a contradiction with $(ii)$.
\end{proof}

The above theorem immediately yields the following upper bound on
$\gamma(G\diamond H).$

\begin{theorem}
\label{Cor_basicUpper}For any two graphs $G$ and $H$
\[
\gamma(G\diamond H)\leq\min\{\gamma(G)+\gamma(H)-1,\gamma_{t}(\overline
{G})+\gamma_{t}(\overline{H})-1\}.
\]

\end{theorem}

\begin{proof}
Let us first assume that $2\leq\gamma_{t}(\overline{G})\leq\gamma
_{t}(\overline{H})<\infty$.\ Let $D_{G}=\{g_{1},\ldots,g_{k}\}$ be a
$\gamma(G)$-set (resp. $\gamma_{t}(\overline{G})$-set) and $D_{H}%
=\{h_{1},\ldots,h_{t}\}$ be $\gamma(H)$-set (resp. $\gamma_{t}(\overline{H}%
)$-set). Notice that
\[
D=\{(g_{1},h_{1}),(g_{2},h_{1}),(g_{2},h_{2}),(g_{3},h_{2}),(g_{3}%
,h_{3}),\ldots,(g_{k},h_{k}),(g_{k},h_{k+1}),\ldots,(g_{k},h_{t})\},
\]
is a snake in $G\diamond H$ of cardinality $\left\vert D\right\vert
=\gamma(G)+\gamma(H)-1$ (resp. $\left\vert D\right\vert =\gamma_{t}%
(\overline{G})+\gamma_{t}(\overline{H})-1$). Since $\mathrm{proj}_{G}%
(D)=D_{G}$ and $\mathrm{proj}_{H}(D)=D_{H},$ Theorem \ref{Tm79} applied to $D$
immediately yields the result. If $\gamma_{t}(\overline{G})=\infty$ or
$\gamma_{t}(\overline{H})=\infty,$ then the bound $\gamma(G\diamond
H)\leq\gamma_{t}(\overline{G})+\gamma_{t}(\overline{H})-1$ trivially holds,
and $\gamma(G\diamond H)\leq\gamma(G)+\gamma(H)-1$ holds by the same argument
as before.
\end{proof}

\vspace{0.2cm}

We will show in the next section that there are many graphs attaining this
upper bound, but for this we first need some further results. Nevertheless, we
can lower the above bound by one under some special conditions.

\begin{lemma}
\label{Lemma_snake}Let $D=\{(x_{i},y_{i}):i\in\lbrack q]\}$ be a snake in
$G\diamond H$ and $(g,h)$ a vertex of $G\diamond H$ which is not dominated by
$D.$ If $x_{1}\in N_{G}[g],$ then $N_{G}[g]\cap\mathrm{proj}_{G}%
(D)=\mathrm{proj}_{G}(D)$ and $N_{H}[h]\cap\mathrm{proj}_{H}(D)=\emptyset.$
The mirror claim also holds.
\end{lemma}

\begin{proof}
Since $(g,h)$ is not dominated by $(x_{1},y_{1}),$ Observation
\ref{Obs_adjacency} implies that either $x_{1}\in N_{G}[g]$ and $y_{1}%
\not \in N_{H}[h]$ or $x_{1}\notin N_{G}[g]$ and $y_{1}\in N_{H}[h]$. Notice
that $x_{1}\in N_{G}[g]$ implies $y_{1}\not \in N_{H}[h].$ Since $D$ is a
snake, we have either $x_{2}=x_{1}$ or $y_{2}=y_{1}.$ If $x_{2}=x_{1},$ then
$x_{2}\in N_{G}[g]$ implies $y_{2}\not \in N_{H}[h].$ If $y_{2}=y_{1},$ then
$y_{2}\not \in N_{H}[h]$ implies $x_{2}\in N_{G}[g].$ Applying the same
argument on $i\in\{3,\ldots,q\}$ yields the claim.
\end{proof}

\bigskip

A snake $D=\{(x_{i},y_{i}):i\in\lbrack q]\}$ for $q\geq2$ is a $G$%
\emph{-snake} if $x_{1}=x_{2}$ and it is an $H$\emph{-snake} if $y_{1}=y_{2}.$
A snake $D$ with $q=1$ is considered to be both a $G$-snake and an $H$-snake.
Notice that for any $i\in\lbrack q],$ a snake $D$ can be divided into two
snakes starting at $(x_{i},y_{i}),$ these are $D_{i}^{-}=\{(x_{i-j+1}%
,y_{i-j+1}):j\in\lbrack i]\}$ and $D_{i}^{+}=\{(x_{i+j-1},y_{i+j-1}%
):j\in\lbrack q+1-i]\}.$ A vertex $(x_{i},y_{i})$ of a snake is a
$G$\emph{-corner} if $D_{i}^{-}$ is a $G$-snake and $D_{i}^{+}$ an $H$-snake.
Similarly, $(x_{i},y_{i})$ is an $H$\emph{-corner} if $D_{i}^{-}$ is an
$H$-snake and $D_{i}^{+}$ a $G$-snake. For example, in the snake $D$ from
Figure \ref{Fig04}, the vertex $(g_{5},h_{4})$ is a $G\,$-corner and the
vertex $(g_{5},h_{2})$ is an $H$-corner.

\begin{theorem}
\label{Tm_corner}Let $D=\{(x_{j},y_{j}):j\in\lbrack q]\}$ be a snake in
$G\diamond H$ which is a dominating set. Let $(g,h)$ be a vertex of $G\diamond
H$ which is dominated only by $(x_{i},y_{i})$ from $D.$ If $(x_{i},y_{i})$ is
a $G$-corner, then $N_{G}[g]\cap\mathrm{proj}_{G}(D)=\{x_{1},\ldots,x_{i}\}$
and $N_{H}[h]\cap\mathrm{proj}_{H}(D)=\{y_{i},\ldots,y_{q}\}$ or $N_{G}%
[g]\cap\mathrm{proj}_{G}(D)=\{x_{i+1},\ldots,x_{q}\}$ and $N_{H}%
[h]\cap\mathrm{proj}_{H}(D)=\{y_{1},\ldots,y_{i-1}\}.$ The mirror claim also holds.
\end{theorem}

\begin{proof}
Since $(g,h)$ is dominated by $(x_{i},y_{i})\in D$, Observation
\ref{Obs_adjacency} implies that either $x_{i}\in N_{G}[g]$ and $y_{i}\in
N_{H}[h]$ or $x_{i}\notin N_{G}[g]$ and $y_{i}\notin N_{H}[h]$. Denote by
$D_{i-1}^{-}$ (resp. $D_{i+1}^{+}$) the snake obtained from $D_{i}^{-}$ (resp.
$D_{i}^{+}$) by removing the first element $(x_{i},y_{i})$ from it.

\vspace{0.2cm}

Assume first that $x_{i}\in N_{G}[g]$ and $y_{i}\in N_{H}[h].$ Since
$(x_{i},y_{i})$ is a $G$-corner, we have that $x_{i-1}\in N_{G}[g]$ and
$y_{i+1}\in N_{H}[h]$, so we can apply Lemma \ref{Lemma_snake} to the snakes
$D_{i-1}^{-}$ and $D_{i+1}^{+}$ to get the result.

Assume next that $x_{i}\notin N_{G}[g]$ and $y_{i}\notin N_{H}[h].$ Since
$(x_{i},y_{i})$ is a $G$-corner, then $x_{i-1}\not \in N_{G}[g]$ and
$y_{i+1}\not \in N_{H}[h]$, consequently, $y_{i-1}\in N_{H}[h]$ and
$x_{i+1}\in N_{G}[g]$, so we can again apply Lemma \ref{Lemma_snake} to the
snakes $D_{i-1}^{-}$ and $D_{i+1}^{+}$ to get the result.
\end{proof}

The above theorem now easily yields the following corollary.

\begin{corollary}
Let $G$ and $H$ be two graphs, let $D_{G}=\{g_{1},\ldots,g_{k}\}$ be a
$\gamma(G)$-set \emph{(}resp. $\gamma_{t}(\overline{G})$-set\emph{)} and
$D_{H}=\{h_{1},\ldots,h_{t}\}$ a $\gamma(H)$-set \emph{(}resp. $\gamma
_{t}(\overline{H})$-set\emph{)}. Suppose that there exists $i\leq\min\{k,t\}$
such that for every $g\in V(G)$ it holds that $|N_{G}[g]\cap D_{G}%
|\notin\{i,k-i\},$ or that the mirror condition holds. Then, $\gamma(G\diamond
H)\leq\gamma(G)+\gamma(H)-2$ \emph{(}resp. $\gamma(G\diamond H)\leq\gamma
_{t}(\overline{G})+\gamma_{t}(\overline{H})-2$\emph{)}.
\end{corollary}

\begin{proof}
Let us consider the set
\[
D=\{(g_{1},h_{1}),(g_{2},h_{1}),(g_{2},h_{2}),(g_{3},h_{2}),(g_{3}%
,h_{3}),\ldots,(g_{k},h_{k-1}),(g_{k},h_{k}),\ldots,(g_{k},h_{t})\}.
\]
Notice that $D$ is a snake and $\left\vert D\right\vert =k+t-1.$ Since
$\mathrm{proj}_{G}(D)=D_{G}$ is a $\gamma(G)$-set (resp. $\gamma_{t}%
(\overline{G})$-set) and $\mathrm{proj}_{H}(D)=D_{H}$ is a $\gamma(H)$-set
(resp. $\gamma_{t}(\overline{H})$-set), Theorem \ref{Tm79} implies that $D$ is
a dominating set in $G\diamond H$. We claim that $D\backslash\{(g_{i}%
,h_{i})\}$ is also a dominating set in $G\diamond H.$ Assume to the contrary
that $D\backslash\{(g_{i},h_{i})\}$ is not a dominating set in $G\diamond H$.
Hence, there must exist $(g,h)$ in $G\diamond H$ which is dominated only by
$(g_{i},h_{i})$ of $D.$ Notice that $(g_{i},h_{i})$ is a $G$-corner, so
Theorem \ref{Tm_corner} implies that $\left\vert N_{G}[g]\cap\mathrm{proj}%
_{G}(D)\right\vert \in\{i,k-i\},$ a contradiction.
\end{proof}

\smallskip

To show that Theorems \ref{Lobound} and \ref{Cor_basicUpper} are sharp, we
consider any graph $H$ with a universal vertex $h$. By Corollary
\ref{Cor_sufficient} we have $\gamma(G\diamond H)\geq\gamma(G)$ and from
Theorem \ref{Cor_basicUpper} $\gamma(G\diamond H)\leq\gamma(G)$ follows.
Whence we obtain the next result.

\begin{corollary}
\label{univ} If $G$ and $H$ are two graphs such that $H$ has a universal
vertex, then $\gamma(G\diamond H)=\gamma(G)$.
\end{corollary}

\noindent Since $\gamma(G)$ is not bounded, clearly also $\gamma(G\diamond H)$
is not bounded. More families of graphs with unbounded $\gamma(G\diamond H)$
will be presented at the end of this section.

The first of the exact result in the next corollary is a direct consequence of
(\ref{complete}) and the results on $\gamma(G\boxtimes K_{n})$ and
$\gamma(G\circ K_{n})$ from the literature, but both of them also follow from
Corollary \ref{univ}.

\begin{corollary}
\label{Kn} For a graph $G$ and a positive integer $n$, the following
equalities hold

\begin{itemize}
\item[\emph{(i)}] $\gamma(G\diamond K_{n})=\gamma(G)$;

\item[\emph{(ii)}] $\gamma(G\diamond K_{1,n})=\gamma(G)$, where $K_{1,n}$ is a
complete bipartite graph.
\end{itemize}
\end{corollary}

It is easy to see that $D=D_{G}\times\{h\}$ is a dominating set of $G\diamond
H$ for any $\gamma(G)$-set $D_{G}$ and a universal vertex $h$ of $H$. But, if
$h$ is not a universal vertex in $H$, as we will see in the next proposition,
$D_{G}$ needs to be an SDCTD-set of $G$ to yield a dominating set of
$G\diamond H$.

\begin{proposition}
\label{SDCTD} If $G$ and $H$ are two graphs, then $\gamma(G\diamond H)\leq
\min\{\bar{\gamma}(G),\bar{\gamma}(H)\}$.
\end{proposition}

\begin{proof}
Let $D_{1}$ be a $\bar{\gamma}(G)$-set, let $h_{1}$ be a vertex from $H$ and
let $D=D_{1}\times\{h_{1}\}$. Let us see that any vertex $(g,h)\in V(G\diamond
H)\setminus D$ is dominated by $D$. If $h\in N_{H}[h_{1}]$, then $(g,h)$ is
dominated by $D$ since $D_{1}$ is a dominating set in $G$. If $h\notin
N_{H}[h_{1}]$, then there exists $g^{\prime}\in D_{1}$ which is not adjacent
to $g$ since $D_{1}$ is a SDCTD set in $G$, so $(g,h)$ is adjacent to
$(g^{\prime},h_{1})$ by a co-direct edge. Consequently, $\gamma(G\diamond
H)\leq|D|=|D_{1}|=\bar{\gamma}(G)$. By symmetric arguments, we can see that
also $\gamma(G\diamond H)\leq\bar{\gamma}(H)$ holds.
\end{proof}

If $G$ is a graph with $\mathrm{diam}(G)\geq3$, then $\bar{\gamma}(G)$ is
close to $\gamma(G)$. More precisely,
\[
\gamma(G)\leq\bar{\gamma}(G)\leq\gamma(G)+2=\gamma(G)+\gamma_{t}(\overline
{G}),
\]
since $D\cup\{g,g^{\prime}\}$ is a SDCTD set for a $\gamma(G)$-set $D$ and for
$g,g^{\prime}\in V(G)$ such that $d_{G}(g,g^{\prime})\geq3$. Hence,
Proposition \ref{SDCTD} yields the next result.

\begin{corollary}
\label{Cor_13}If $G$ and $H$ are two graphs such that $\mathrm{diam}(G)\geq3$,
then $\gamma(G\diamond H)\leq\gamma(G)+2$.
\end{corollary}

It is easy to observe that, if $D$ is a dominating set in $G$ and it contains
two vertices $g_{1},g_{2}\in D$ such that $d_{G}(g_{1},g_{2})\geq3$, then $D$
is an SDCTD set. In particular, if $G$ is an ECD graph with $\gamma(G)\geq2$,
it contains a dominating set satisfying that condition.

\begin{corollary}
\label{Prop_case3}Let $G$ and $H$ be two graphs. If $D_{G}$ is a $\gamma
(G)$-set in $G$ and there exist $g_{1},g_{2}\in D_{G}$ such that $d_{G}%
(g_{1},g_{2})\geq3$, then $\gamma(G\diamond H)\leq\gamma(G)$.
\end{corollary}

\begin{proof}
The result follows from Proposition \ref{SDCTD} since $D_{G}$ is a SDCTD set
of $G$.
\end{proof}

\begin{corollary}
\label{ECD} If $G$ is an \emph{ECD} graph wuith $\gamma(G)\geq2$ and $H$ is
any graph, then $\gamma(G\diamond H)\leq\gamma(G)$.
\end{corollary}

We also have the following corollary.

\begin{corollary}
\label{diam5} If $G$ is a graph such that $\mathrm{diam}(G)\geq5$, then
$\gamma(G\diamond H)\leq\gamma(G)$.
\end{corollary}

\begin{proof}
If we have two vertices $g_{1},g_{2}$ such that $d_{G}(g_{1},g_{2})\geq5$, for
any dominating set $D_{G}$, since $D_{G}\cap N[g_{1}]\neq\emptyset$ and
$D_{G}\cap N[g_{2}]\neq\emptyset$, there exist two vertices $u_{1},u_{2}\in
D_{G}$ such that $d_{G}(u_{1},u_{2})\geq3$.
\end{proof}

Some of well known ECD graphs are the cube $Q_{3}$, the cube minus a vertex
$Q_{3}^{-}$, cycles $C_{3k}$, $k>1$, paths $P_{t}$, $t>3$, and a complete
bipartite graph minus an edge $K_{m,n}^{-}$, $m,n>1$. The next result follows
directly from Corollary \ref{ECD} because the considered graphs have no
universal vertex and as we will see in Proposition \ref{one}, $\gamma
(G\diamond H)=1$ if and only if $G$ and $H$ has an universal vertex.

\begin{corollary}
\label{examples} For any graph $G$ and integers $k,m,n>1$ and $t>3$ we have

\begin{itemize}
\item[\emph{(i)}] $\gamma(G\diamond Q_{3})=2$;

\item[\emph{(ii)}] $\gamma(G\diamond Q_{3}^{-})=2$;

\item[\emph{(iii)}] $\gamma(G\diamond K_{m,n}^{-})=2$;

\item[\emph{(iv)}] $\gamma(G\diamond C_{3k})\leq k$;

\item[\emph{(v)}] $\gamma(G\diamond P_{t})\leq\left\lceil \frac{t}%
{3}\right\rceil $.
\end{itemize}
\end{corollary}

\begin{figure}[h]
\centering\includegraphics[scale=0.4]{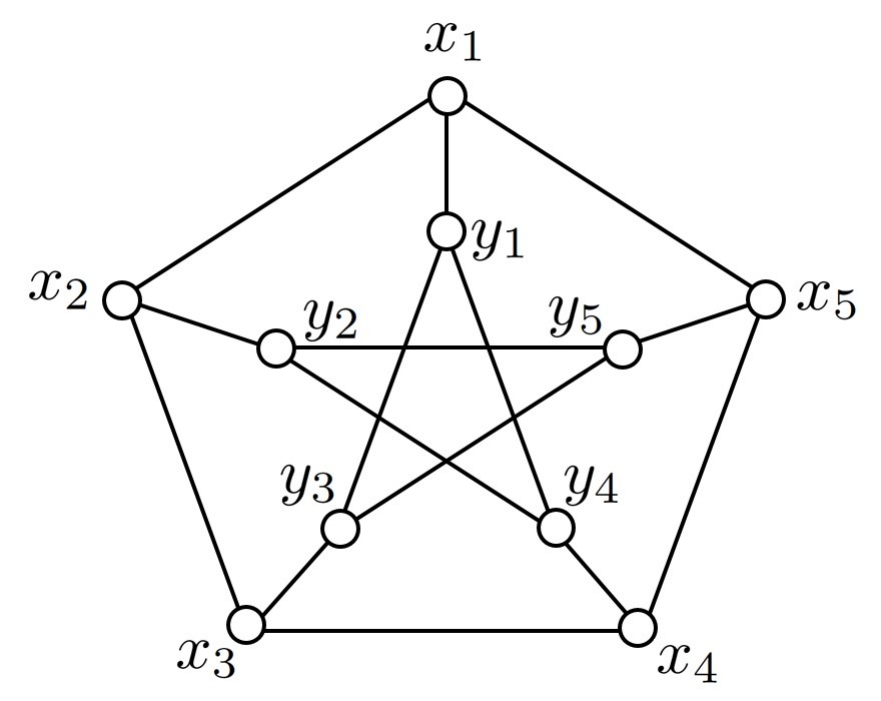}\caption{The Petersen graph
$P.$}%
\label{Fig_Petersen}%
\end{figure}

We see that the upper bound of Proposition \ref{SDCTD} is sharp for many
cases, however, there are graphs $G$, with diameter 2 and without universal
vertex, such that $\gamma(G\diamond G)<\bar{\gamma}(G)$. For instance, if we
consider the Petersen graph $P$ labeled as in Figure \ref{Fig_Petersen}, we
know that $\gamma(P)=3$ and that an open neighborhood of every vertex, say
$N_{P}(x_{1})$, represents a $\gamma(P)$-set and there are no other
$\gamma(P)$-sets. However, $N_{P}(x_{1})$ is not a SDCTD-set of $P$. It is
easy to check that, for instance, $\{x_{2},y_{1},y_{5},x_{4}\}$ forms a
SDCTD-set of $P$ and $\bar{\gamma}(P)=4$. By Proposition \ref{SDCTD} we have
\[
\gamma(P\diamond P)\leq4.
\]
On the other hand, one can observe that $\{(x_{1},y_{5}),(x_{4},y_{1}),$
$(y_{5},x_{3})\}$ is a $\gamma(P\diamond P)$-set, this means, $\gamma
(P\diamond P)\leq3$. Since $\mathrm{diam}(P)=2$, by Proposition \ref{total2},
we know that $\gamma_{t}(\overline{P})\geq3$, so $3=\min\{\gamma(P),\gamma
_{t}(\overline{P})\}\leq\gamma(P\diamond P)$. Therefore, $\gamma(P\diamond
P)=3$.

Since $\bar{\gamma}(C_{k})=3$ for $k\in\{4,5\}$ and $\bar{\gamma}%
(C_{k})=\left\lceil \frac{k}{3}\right\rceil $ for $k\geq6$, as a consequence
of Proposition \ref{SDCTD}, we have the following result.

\begin{corollary}
\label{peter} For any graph $G$ we have

\begin{itemize}
\item[\emph{(i)}] $\gamma(G\diamond C_{k})\leq3$ for $k\in\{4,5\}$;

\item[\emph{(ii)}] $\gamma(G\diamond C_{6})=2$;

\item[\emph{(iii)}] $\gamma(G\diamond C_{k})\leq\left\lceil \frac{k}%
{3}\right\rceil $ for $k\geq7$;

\item[\emph{(iv)}] $\gamma(G\diamond P)\leq4$.
\end{itemize}
\end{corollary}

Let us now give an example of graph classes $G$ and $H$ for which
$\gamma(G\diamond H)$ is unbounded, even when none of them has a universal
vertex. Denote by $K_{2k}^{-}$ a graph $K_{2k}-M$, $k\geq3$, where $M$ is a
perfect matching in $K_{2k}$. Theorem \ref{Lobound} and Proposition
\ref{SDCTD} imply the following result.

\begin{corollary}
The following equalities hold

\begin{itemize}
\item[\emph{(i)}] $\gamma(K_{2k}^{-}\diamond P_{6k})=2k$ for any $k\geq3$;

\item[\emph{(ii)}] $\gamma(P_{6k}\diamond\overline{P_{4k}})=2k$ for any
$k\geq2$.
\end{itemize}
\end{corollary}

\begin{proof}
These two equalities are satisfied because
\[
2k=\min\{2k,2k\}=\min\{\gamma(P_{6k}),\gamma_{t}(\overline{K_{2k}^{-}}%
)\}\leq\gamma(K_{2k}^{-}\diamond P_{6k})\leq\bar{\gamma}(P_{6k})=2k
\]
and
\[
2k=\min\{2k,2k\}=\min\{\gamma(P_{6k}),\gamma_{t}(P_{4k})\}\leq\gamma
(P_{6k}\diamond\overline{P_{4k}})\leq\bar{\gamma}(P_{6k})=\gamma(P_{6k})=2k.
\]

\end{proof}

Notice that in both families of the above corollary one factor is diameter two
graph and the other has unbounded diameter. Even with two graphs of diameter
two, we can find examples where $\gamma(G\diamond H)$ is not bounded. There
exist graphs $G$ such that $\mathrm{diam}(G)=2$ and $\gamma(G)\geq\frac
{\gamma_{t}(G)}{2}\geq2k$ (see \cite{Du}). For such graphs we have
$\gamma(G\diamond\overline{P_{4k}})\geq\min\{\gamma(G),\gamma_{t}%
(P_{4k})\}\geq2k$.

\section{Graphs with small values of $\gamma(G\diamond H)$}

In this section we describe modular products with small domination number. We
start by completely characterizing pairs of graphs $G$, $H$ such that the
domination number of their modular product equals $1$ or $2.$

\begin{proposition}
\label{one} For any graphs $G$ and $H$, $\gamma(G\diamond H)=1$ if and only if
$\gamma(G)=1=\gamma(H)$.
\end{proposition}

\begin{proof}
If $D=\{(g_{1},h_{1})\}$ is a $\gamma(G\diamond H)$-set, since any vertex
$(g,h_{1})$ must be dominated by $D$, $g_{1}$ must be an universal vertex in
$G$, that is, $D_{G}=\{g_{1}\}$ is a dominating set in $G$. Analogously, we
can have the same with $D_{H}=\{h_{1}\}$. Finally, by Corollary \ref{univ}, we
have that $\gamma(G)=1=\gamma(H)$ implies $\gamma(G\diamond H)=1.$
\end{proof}

\begin{theorem}
\label{two} For any graphs $G$ and $H$, $\gamma(G\diamond H)=2$ if and only if
one of the following conditions holds

\vspace{0.2cm}

\begin{itemize}
\item[\emph{(i)}] $\gamma(G)+\gamma(H)=3$;

\vspace{0.2cm}

\item[\emph{(ii)}] $G$ or $H$ has an \emph{ECD} set of size 2.
\end{itemize}
\end{theorem}

\begin{proof}
If $D_{G}=\{g_{1},g_{2}\}$ is a $\gamma(G)$-set and $h_{1}$ is a universal
vertex in $H$, then $\gamma(G\diamond H)=2$ by Corollary \ref{univ}. If
$D_{G}=\{g_{1},g_{2}\}$ is an ECD set of $G$, then $\gamma(G\diamond H)\leq2$
by Corollary \ref{ECD}. The equality follows from Proposition \ref{one}. The
symmetric conditions on factors $G$ and $H$ conclude this direction.
\vspace{0.2cm}

Now, suppose that $\gamma(G\diamond H)=2$, and consider the following two cases.

\vspace{0.2cm}

\noindent\textbf{Case 1:} $D=\{(g_{1},h_{1}),(g_{2},h_{1})\}$\emph{ is a
}$\gamma(G\diamond H)$\emph{-set.} Since $(g,h_{1})$ must be dominated for any
$g\in V(G)$, we deduce that $D_{G}=\{g_{1},g_{2}\}$ is a dominating set in
$G$. If $h_{1}$ is a universal vertex, then we have (i). If $h_{1}$ is not a
universal vertex, there exists $h\in V(H)$ which is not adjacent to $h_{1}$
and $g_{1}g_{2}\notin E(G)$ to dominate $(g_{1},h)$. Finally, if $g\in
N_{G}(g_{1})$, to dominate $(g,h)$, we need $g\notin N_{G}(g_{2})$. Therefore,
$N_{G}(g_{1})\cap N_{G}(g_{2})=\emptyset$ and, consequently, $d_{G}%
(g_{1},g_{2})=3$ and $D_{G}$ is an ECD set of $G$. If $D=\{(g_{1}%
,h_{1}),(g_{1},h_{2})\}$ is a $\gamma(G\diamond H)$-set, then we proceed symmetrically.

\vspace{0.2cm}

\noindent\textbf{Case 2:} $D=\{(g_{1},h_{1}),(g_{2},h_{2})\}$\emph{ is a
}$\gamma(G\diamond H)$\emph{-set where }$g_{1}\neq g_{2}$\emph{ and }%
$h_{1}\neq h_{2}$\emph{.} To dominate the vertex $(g_{1},h_{2})$ we need
$g_{1}g_{2}\in E(G)$ or $h_{1}h_{2}\in E(H)$. By symmetry, we suppose for
instance that $h_{1}h_{2}\in E(H)$.

\vspace{0.2cm}

\noindent\textbf{Case 2.1:} $g_{1}$\emph{ }(\emph{or }$g_{2}$)\emph{ is not a
universal vertex.} If we take $g\notin N_{G}[g_{1}]$, then $(g,h_{1})$ must be
dominated by $(g_{2},h_{2})$ by a Cartesian or a direct edge. Consequently,
$g\in N_{G}[g_{2}]$ and $\{g_{1},g_{2}\}$ is a dominating set in $G$.
Moreover, for every $h\in N_{H}(h_{1})$, since $(g,h)$ must be dominated by
$(g_{2},h_{2})$ by a Cartesian or a direct edge, we have that $h\in
N_{H}[h_{2}]$. Hence, $N_{H}[h_{1}]\subseteq N_{H}[h_{2}]$. If $h_{2}$ is a
universal vertex in $H$, then we have (i). If $N_{H}[h_{2}]\neq V(H)$, since
$V(G)\times(V(H)\setminus N_{H}[h_{2}])$ is dominated by $D$, then
$N_{G}[g_{1}]\cap N_{G}[g_{2}]=\emptyset$, thus $d_{G}(g_{1},g_{2})=3$,
$\{g_{1},g_{2}\}$ is an ECD set and we have (ii).

\vspace{0.2cm}

\noindent\textbf{Case 2.2:} $g_{1}$\emph{ and }$g_{2}$\emph{ are universal
vertices in }$G$\emph{.} Since any vertex $(g,h)$ must be dominated by $D$,
$h$ must be adjacent to $h_{1}$ or to $h_{2}$. If there exists a universal
vertex in $H$, we have a contradiction with Proposition \ref{one}. Therefore,
$\{h_{1},h_{2}\}$ is a minimum dominating set in $H$ and we have (i).

\vspace{0.2cm}

Hence, in each of the two cases we have established that one of the conditions
(i) and (ii) holds, so we are finished.
\end{proof}

Next, we give a characterization of graphs $G$ and $H$ such that the
domination number of their modular product is at least $3.$

\begin{corollary}
\label{biggertwo} For any graphs $G$ and $H$, $\gamma(G\diamond H)\geq3$ if
and only if the following conditions hold

\begin{itemize}
\vspace{0.2cm}

\item[\emph{(i)}] $\gamma(G)+\gamma(H)\geq4$;

\vspace{0.2cm}

\item[\emph{(ii)}] neither $G$ nor $H$ has an ECD set of size 2.
\end{itemize}
\end{corollary}

Since we have characterized all graphs $G$ and $H$ such that the domination
number of their modular product is one or two, from now on, we suppose that
$G$ and $H$ satisfy the conditions given in Corollary \ref{biggertwo}. Next,
we are going to characterize graphs $G$ and $H$ such that $\gamma(G\diamond
H)=3$.

It was proved in \cite[Theorem 3.5]{NoRa} that two graphs $G$ and $H$ with
packing numbers at least two have $\gamma(G\diamond H)\leq4$. Since every
graph with packing number at least two has diameter at least three, we switch
to diameters and improve that bound by one. For this we first need the
following lemma.

\begin{lemma}
\label{red} Let $G$ and $H$ be two graphs. If $g_{1},g_{2}\in V(G)$,
$g_{1}\neq g_{2}$, $h_{1},h_{2}\in V(H)$, $h_{1}\neq h_{2}$, then
\[
N_{G\diamond H}[\{(g_{1},h_{1}),(g_{1},h_{2}),(g_{2},h_{1}),(g_{2}%
,h_{2})\}]=N_{G\diamond H}[\{(g_{1},h_{1}),(g_{1},h_{2}),(g_{2},h_{1})\}].
\]

\end{lemma}

\begin{proof}
Let us see that any vertex $(g,h)\in N_{G\diamond H}[(g_{2},h_{2})]$ belongs
to $N_{G\diamond H}[\{(g_{1},h_{1}),(g_{1},h_{2}),$ $(g_{2},h_{1})\}]$.
Firstly, we suppose that $(g,h)\in N_{G}[g_{2}]\times N_{H}[h_{2}]$. If $g\in
N_{G}[g_{1}]$, then $(g,h)$ is equal or adjacent to $(g_{1},h_{2})$. If $h\in
N_{H}[h_{1}]$, then $(g,h)$ is equal or adjacent to $(g_{2},h_{1})$. If
$g\notin N_{G}[g_{1}]$ and $h\notin N_{H}[h_{1}]$, then $(g,h)$ is adjacent to
$(g_{1},h_{1})$.

\vspace{0.2cm}

Secondly, we suppose that $g\notin N_{G}[g_{2}]$ and $h\notin N_{H}[h_{2}]$.
If $g\notin N_{G}[g_{1}]$, then $(g,h)$ is adjacent to $(g_{1},h_{2})$. If
$h\notin N_{H}(h_{1})$, then $(g,h)$ is adjacent to $(g_{2},h_{1})$. If $g\in
N_{G}[g_{1}]$ and $h\in N_{H}[h_{1}]$, then $(g,h)$ is equal or adjacent to
$(g_{1},h_{1})$.
\end{proof}

We can now proceed by giving the following sufficient condition for the
domination number of the modular product of graphs $G$ and $H$ to be at most
$3.$

\begin{proposition}
\label{diameter3} If $G$ and $H$ are graphs with diameter at least three, then
$\gamma(G\diamond H)\leq3$.
\end{proposition}

\begin{proof}
Let $g_{1},g_{2}\in V(G)$ and $h_{1},h_{2}\in V(H)$ be four vertices such that
$d_{G}(g_{1},g_{2})\geq3$ and $d_{H}(h_{1},h_{2})\geq3$. For any vertex
$(g,h)\in V(G\diamond H)$, there exist $i,j\in\lbrack2]$ such that $g\notin
N_{G}(g_{i})$ and $h\notin N_{H}(h_{j})$, so $(g,h)$ is adjacent to
$(g_{i},h_{j})$ by a co-direct edge. Therefore, the set $D=\{(g_{1}%
,h_{1}),(g_{1},h_{2}),(g_{2},h_{1}),(g_{2},h_{2})\}$ is a dominating set in
$G\diamond H$. By Lemma \ref{red}, we conclude that $\gamma(G\diamond H)\leq3$.
\end{proof}

The above proposition together with Corollary \ref{biggertwo} yields the
following condition under which $\gamma(G\diamond H)=3.$

\begin{corollary}
If $G$ and $H$ are graphs with diameter at least three and they do not have
ECD sets with size 2, then $\gamma(G\diamond H)=3$.
\end{corollary}

Another condition on graphs $G$ and $H$ under which $\gamma(G\diamond H)=3$ is
given in the following proposition, and it stems from Corollary
\ref{biggertwo} and Lemma \ref{red}.

\begin{proposition}
\label{2} If $G$ and $H$ are two graphs such that $\gamma(G)=2=\gamma(H)$ and
they do not have ECD sets with size 2, then $\gamma(G\diamond H)=3$.
\end{proposition}

\begin{proof}
If $\{g_{1},g_{2}\}$ is a $\gamma(G)$-set and $\{h_{1},h_{2}\}$ is a
$\gamma(H)$-set, it is clear that
\[
D=\{(g_{1},h_{1}),(g_{1},h_{2}),(g_{2},h_{1}),(g_{2},h_{2})\}
\]
is a dominating set in $G\diamond H$. By Corollary \ref{biggertwo} and Lemma
\ref{red}, we get the result.
\end{proof}



Next we characterize all pairs of graphs $G$ and $H$ with $\gamma(G\diamond
H)=3$. For this we need the following notation. Let $G$ be a graph and
$D=\{g_{1},\ldots,g_{q}\}\subseteq V(G).$ For a set $I\subseteq\lbrack q]$
with $q$ being fixed, let us denote $I^{c}=[q]\backslash I,$ and we define
\[
A_{G}(I,D)=\{g\in G:N_{G}[g]\cap D=\{g_{i}:i\in I\}\}.
\]
We further define $a_{G}(I,D)=1$ if $A_{G}(I,D)\not =\emptyset,$ and
$a_{G}(I,D)=0$ otherwise.

\begin{theorem}
\label{dom3} Let $G$ and $H$ be two graphs such that $\gamma(G\diamond
H)\geq3$. Then, $\gamma(G\diamond H)=3$ if and only if at least one of the
following conditions holds

\vspace{0.1cm}

\begin{itemize}
\item[\emph{(i)}] $\gamma(G)+\gamma(H)=4$;

\vspace{0.2cm}

\item[\emph{(ii)}] $\bar{\gamma}(G)=3$ or $\bar{\gamma}(H)=3$;

\vspace{0.2cm}

\item[\emph{(iii)}] $\mathrm{diam}(G)\geq3$ and $\mathrm{diam}(H)\geq3$;

\vspace{0.2cm}

\item[\emph{(iv)}] there exist two sets $D_{G}=\{g_{1},g_{2},g_{3}\}\subseteq
V(G)$ and $D_{H}=\{h_{1},h_{2},h_{3}\}\subseteq V(H)$ such that $a_{G}%
(I,D_{G})+a_{H}(I^{c},D_{H})\leq1$ for every $I\subseteq\lbrack3].$
\end{itemize}
\end{theorem}

\begin{proof}
Let us first show that each of the conditions (i)-(iv) implies $\gamma
(G\diamond H)=3.$

\vspace{0.2cm}

\noindent(i) If $\gamma(G)+\gamma(H)=4$, then the result follows either by
Corollary \ref{univ}, when $G$ or $H$ has a universal vertex, or by
Proposition \ref{2}.

\vspace{0.2cm}

\noindent(ii) If $\bar{\gamma}(G)=3$ or $\bar{\gamma}(H)=3$, then Proposition
\ref{SDCTD} implies $\gamma(G\diamond H)\leq3$.

\vspace{0.2cm}

\noindent(iii) If $\mathrm{diam}(G)\geq3$ and $\mathrm{diam}(H)\geq3$, then
Proposition \ref{diameter3} implies $\gamma(G\diamond H)\leq3.$

\vspace{0.2cm}

\noindent(iv) Assume that there exist two sets $D_{G}=\{g_{1},g_{2}%
,g_{3}\}\subseteq V(G)$ and $D_{H}=\{h_{1},h_{2},h_{3}\}\subseteq V(H)$ such
that $a_{G}(I,D_{G})+a_{H}(I^{c},D_{H})\leq1$ for every $I\subseteq\lbrack3],$
and let
\[
D=\{(g_{1},h_{1}),(g_{2},h_{2}),(g_{3},h_{3})\}.
\]
It is sufficient to show that $D$ is a dominating set in $G\diamond H.$ Let
$(g,h)$ be any vertex in $G\diamond H,$ we wish to show that $(g,h)$ is
dominated by at least one vertex of $D.$

Let us first assume $g\in N_{G}[g_{i}]$ for every $i\in\lbrack3].$ This
implies that $g\in A_{G}(I,D_{G})$ for $I=[3],$ i.e $a_{G}(I,D_{G})=1.$ Since
$a_{G}(I,D_{G})+a_{H}(I^{c},D_{H})\leq1,$ this further implies $A_{H}%
(I^{c},D_{H})=\emptyset.$ Hence, there exists $h_{i}\in D_{H}$ such that $h\in
N_{G}[h_{i}],$ so Observation \ref{Obs_adjacency} implies that $(g,h)$ is
dominated by $(g_{i},h_{i})\in D.$

Let us next assume that $g\not \in N_{G}[g_{i}]$ for precisely one
$i\in\lbrack3],$ say $g\not \in N_{G}[g_{3}].$ This means that $g\in
A_{G}(I,D_{G})$ for $I=\{1,2\},$ i.e. $a_{G}(I,D_{G})=1.$ We conclude that
$a_{H}(I^{c},D_{H})=0,$ which implies that either $h\not \in N_{G}[h_{3}]$ in
which case $(g,h)$ is dominated by $(g_{3},h_{3}),$ or $h\in N_{G}[g_{1}]\cup
N_{G}[g_{2}]$ which implies that $(g,h)$ is dominated by at least one of the
vertices $(g_{1},h_{1})$ and $(g_{2},h_{2}).$

Let us further assume that $g\in N_{G}[g_{i}]$ for precisely one $i\in
\lbrack3],$ say $g\in N_{G}[g_{3}].$ This means that $g\in A_{G}(I,D_{G})$ for
$I=\{3\},$ i.e. $a_{G}(I,D_{G})=1.$ Hence, it must hold $a_{H}(I^{c}%
,D_{H})=0,$ which implies that $h\in N_{H}[h_{3}]$ or $h\not \in N_{G}[h_{1}]$
or $h\not \in N_{G}[h_{2}],$ which means that $(g,h)$ is dominated by
$(g_{3},h_{3})$ or $(g_{1},h_{1})$ or $(g_{2},h_{2}),$ respectively.

Let us finally assume $g\not \in N_{G}[g_{i}]$ for every $i\in\lbrack3].$ This
implies that $g\in A_{G}(I,D_{G})$ for $I=\emptyset,$ i.e $a_{G}(I,D_{G})=1.$
Since $a_{G}(I,D_{G})+a_{H}(I^{c},D_{H})\leq1,$ this further implies
$A_{H}(I^{c},D_{H})=\emptyset.$ Hence, there exists $h_{i}\in D_{H}$ such that
$h\not \in N_{G}[h_{i}],$ so $(g,h)$ is dominated by $(g_{i},h_{i})\in D.$

We have established that $(g,h)$ is dominated by $D$ in all the possible
cases, so $D$ is a dominating set in $G\diamond H$ which means $\gamma
(G\diamond H)\leq3.$ As the assumption of the theorem is $\gamma(G\diamond
H)\geq3,$ we conclude that $\gamma(G\diamond H)=3$ which establishes the claim.

\medskip

Let us now prove the other direction of the equivalence, i.e. now we assume
that $\gamma(G\diamond H)=3$. By Corollary \ref{biggertwo}, it holds that
$\gamma(G)+\gamma(H)\geq4$ and neither $G$ nor $H$ have an ECD set of size
$2$. By symmetry, we have four possibilities for the minimum dominating set:

\vspace{0.2cm}

\begin{itemize}
\item[(1)] $D_{1}=\{(g_{1},h_{1}),(g_{2},h_{1}),(g_{3},h_{1})\}$;

\item[(2)] $D_{2}=\{(g_{1},h_{1}),(g_{1},h_{2}),(g_{2},h_{1})\}$;

\item[(3)] $D_{3}=\{(g_{1},h_{1}),(g_{2},h_{1}),(g_{3},h_{2})\}$;

\item[(4)] $D_{4}=\{(g_{1},h_{1}),(g_{2},h_{2}),(g_{3},h_{3})\}$.

\vspace{0.2cm}
\end{itemize}

\noindent\textbf{Case 1:} $D_{1}$\emph{ dominates }$G\diamond H$\emph{.} For
every vertex $g\in V(G)$, the vertex $(g,h_{1})$ must be dominated by $D_{1}$,
so $\{g_{1},g_{2},g_{3}\}$ must be a dominating set in G, i.e. $\gamma
(G)\leq3$. Let us now consider the graph $H,$ and let us first assume that
$\gamma(H)=1.$ Notice that Corollary \ref{biggertwo} implies $\gamma(G)=3.$ We
conclude that $\gamma(G)+\gamma(H)=4,$ so we obtained (i).

Let us next assume that $\gamma(H)>1,$ which implies that $H$ does not contain
a universal vertex. Hence, there must exist $h\not \in N_{H}[h_{1}].$ Let
$g\in V(G)$ be any vertex of $G.$ Since $D_{1}$ dominates $(g,h)$ and $h\notin
N_{H}[h_{1}],$ Observation \ref{Obs_adjacency} implies there exists
$i\in\lbrack3]$ such that $g\not \in N_{G}[g_{i}].$ This means that $g$ is
dominated by $g_{i}$ in $\overline{G},$ so $\{g_{1},g_{2},g_{3}\}$ is a total
dominating set in $\overline{G}.$ We conclude that $\{g_{1},g_{2},g_{3}\}$ is
a SDCTD set in $G,$ which implies $\bar{\gamma}(G)\leq3.$ As we assumed that
$\gamma(G\diamond H)\geq3,$ Corollary \ref{SDCTD} implies $\bar{\gamma}%
(G)\geq3,$ so $\bar{\gamma}(G)=3$ and we obtained (ii).

\vspace{0.2cm}

\noindent\textbf{Case 2:} $D_{2}$\emph{ dominates }$G\diamond H$\emph{.} We
suppose first that $\{g_{1},g_{2}\}$ is not a dominating set in $G$. This
implies that there exists $g\notin N_{G}[g_{1}]\cup N_{G}[g_{2}]$. Since
$(g,h_{1})$ is dominated by $D_{2},$ Observation \ref{Obs_adjacency} implies
$h_{1}\not \in N_{H}[h_{2}]$ which means $d_{H}(h_{1},h_{2})\geq2.$ Moreover,
if $d_{H}(h_{1},h_{2})=2,$ then there exists $h\in N_{H}[h_{1}]\cap
N_{H}[h_{2}],$ in which case Observation \ref{Obs_adjacency} implies that
$(g,h)$ is not dominated by $D_{2}$, a contradiction. We conclude that
$d_{H}(h_{1},h_{2})\geq3,$ so $\mathrm{diam}(H)\geq3.$

Let us now consider the graph $H$ and the set $\{h_{1},h_{2}\}.$ If
$\{h_{1},h_{2}\}$ is not a dominating set in $H,$ then by the same argument we
have $\mathrm{diam}(G)\geq3$ and we obtained (iii). On the other hand, if
$\{h_{1},h_{2}\}$ is a dominating set in $H,$ Corrolary \ref{biggertwo}
implies that $\{h_{1},h_{2}\}$ is not an ECD set, which means there exists
$h\in N_{H}[h_{1}]\cap N_{H}[h_{2}].$ This implies $d_{H}(h_{1},h_{2})\leq2,$
a contradiction.

We suppose next that $\{g_{1},g_{2}\}$ is a dominating set in $G$. If
$\{h_{1},h_{2}\}$ is not a dominating set in $H,$ then by the symmetric
argument to the one in the previous two paragraphs we have the claim, so we
may assume that $\{h_{1},h_{2}\}$ is a dominating set in $H.$ Hence,
$\gamma(G)=2=\gamma(H)$ and we have (i).

\vspace{0.2cm}

\noindent\textbf{Case 3:} $D_{3}$\emph{ dominates }$G\diamond H$\emph{.}
Assume first that $\bar{\gamma}(G)\leq3.$ Notice that Corollary \ref{SDCTD}
implies $\bar{\gamma}(G)=3,$ so we obtained (iii). Assume next that
$\bar{\gamma}(G)\geq4$ and we distinguish the following two cases.

\vspace{0.2cm}

\noindent\textbf{Case 3.1:} \emph{The set }$D_{G}=\{g_{1},g_{2},g_{3}\}$\emph{
is a dominating set in }$G$\emph{. }From $\bar{\gamma}(G)\geq4$ it follows
that $\{g_{1},g_{2},g_{3}\}$ is not a SDCTD set, so there exists $g\in V(G)$
such that $g\in N_{G}[g_{i}]$ for every $i\in\lbrack3].$ Let $h$ be any vertex
of $H.$ Since $(g,h)$ is dominated by $D_{3},$ Observation \ref{Obs_adjacency}
implies $h\in N_{G}[h_{i}]$ for at least one $i\in\lbrack2].$ We conclude that
$D_{H}=\{h_{1},h_{2}\}$ is a dominating set in $H,$ so $\gamma(H)\leq2.$ Given
that $\gamma(G)\leq3,$ Corollary \ref{biggertwo} implies that either
$\gamma(G)+\gamma(H)=4$ or $\gamma(G)=3$ and $\gamma(H)=2.$ If $\gamma
(G)+\gamma(H)=4,$ we obtain (i), so let us assume $\gamma(G)=3$ and
$\gamma(H)=2.$ Notice that for $I=\{3\}$ there must exist $g\in A_{G}%
(I,D_{G}),$ otherwise $\{g_{1},g_{2}\}$ would be a dominating set in $G,$
which contradicts $\gamma(G)=3.$ Similarly, there must exist $h\in
A_{H}(\{1\},D_{H}),$ otherwise $\{h_{2}\}$ would be a dominating set in $H,$
which contradicts $\gamma(H)=2.$ Now, Observation \ref{Obs_adjacency} implies
$(g,h)$ is not dominated by $D_{3},$ a contradiction.

\vspace{0.2cm}

\noindent\textbf{Case 3.2:} \emph{The set }$D_{G}=\{g_{1},g_{2},g_{3}\}$\emph{
is not a dominating set in }$G.$ Let $g\notin N_{G}[D_{G}]$. If $d_{H}%
(h_{1},h_{2})\leq2,$ there exists $h\in N_{H}[h_{1}]\cap N_{H}[h_{2}],$ so
Observation \ref{Obs_adjacency} implies $(g,h)$ is not dominated by $D_{3},$ a
contradiction. Hence, let us assume that $d_{H}(h_{1},h_{2})\geq3,$ which
implies $\mathrm{diam}(H)\geq3.$ If $D_{H}=\{h_{1},h_{2}\}$ is a dominating
set in $H,$ Corollary \ref{Prop_case3} implies $\gamma(G\diamond H)\leq
\gamma(H)=2,$ a contradiction. So, let us assume that $D_{H}=\{h_{1},h_{2}\}$
is not a dominating set in $H$ and let $h\not \in N_{H}[D_{H}].$ If there
exists $g\in N_{G}[g_{1}]\cap N_{G}[g_{2}]\cap N_{G}[g_{3}],$ Observation
\ref{Obs_adjacency} implies that $(g,h)$ is not dominated by $D_{3}$, a
contradiction. So, let us assume that $N_{G}[g_{1}]\cap N_{G}[g_{2}]\cap
N_{G}[g_{3}]=\emptyset.$ Recall that $d_{H}(h_{1},h_{2})\geq3,$ which implies
$h_{1}\not \in N_{G}[h_{2}]$ and $h_{2}\not \in N_{G}[h_{1}].$ If for
$I=\{3\}$ there exists $g\in A_{G}(I,D_{G}),$ then Observation
\ref{Obs_adjacency} implies that $(g,h_{1})$ is not dominated by $D_{3}$, a
contradiction. Similarly, if there exists $g\in N_{G}[g_{1}]\cap N_{G}%
[g_{2}],$ then $g\not \in N_{G}[g_{3}],$ so $(g,h_{2})$ is not dominated by
$D_{3},$ a contradiction. We conclude that for $I=\{3\}$ it holds that
$N_{G}[g_{1}]\cap N_{G}[g_{2}]=\emptyset$ and $d_{G}(g_{1},g_{2})\geq3,$ so we
have obtained (iii).

\vspace{0.2cm}

\noindent\textbf{Case 4:} $D_{4}$\emph{ dominates }$G\diamond H$\emph{.} Let
$D_{G}=\{g_{1},g_{2},g_{3}\}$ and $D_{H}=\{h_{1},h_{2},h_{3}\}.$ Assume first
$I=[3],$ so $I^{c}=\emptyset.$ If $a_{G}(I,D_{G})=1,$ this means that there
exists $g\in A_{G}(I,D).$ For such a vertex $g,$ it holds that $g\in
N_{G}[g_{1}]\cap N_{G}[g_{2}]\cap N_{G}[g_{3}].$ For any vertex $h\in V(H),$
it holds that $(g,h)$ is dominated by $D_{4}$ in $G\diamond H.$ Hence,
Observation \ref{Obs_adjacency} implies that there exists $h_{i}\in D_{H}$
such that $h\in N_{H}[h_{i}].$ This implies $h\not \in A_{H}(I^{c},D_{H}),$ so
$A_{H}(I^{c},D_{H})=\emptyset$ which means $a_{H}(I^{c},D_{H})=0.$ We conclude
that $a_{G}(I,D_{G})+a_{H}(I^{c},D_{H})\leq1,$ so we obtained (iv).

Assume next that $I\subseteq\lbrack3]$ such that $\left\vert I\right\vert =2,$
say $I=\{1,2\}.$ Then $I^{c}=\{3\}.$ Again, $a_{G}(I,D_{G})=1$ implies that
there exists $g\in V(G)$ such that $g\in(N_{G}[g_{1}]\cap N_{G}[g_{2}%
])\backslash N_{G}[g_{3}].$ If $a_{H}(I^{c},D_{H})=1,$ then there exists $h\in
N_{H}[h_{3}]\backslash(N_{H}[h_{1}]\cup N_{H}[h_{2}]),$ so according to
Observation \ref{Obs_adjacency} the vertex $(g,h)$ is not dominated by
$D_{4},$ a contradiction. Hence, it must hold that $a_{H}(I^{c},D_{H})=0$, so
we again obtained (iv).

The proof for $I\subseteq\lbrack3]$ with cardinality $1$ or $0$ is obtained as
the mirror condition of the previous two cases.

\vspace{0.2cm}

Hence, in each of the cases we obtained that at least one of the condition
(i)-(iv) holds, so we are done.
\end{proof}

Notice that Theorem \ref{dom3} and Corollary \ref{biggertwo} together yield a
characterization of all graphs $G$ and $H$ such that $\gamma(G\diamond H)=3.$

Let us consider the bound $\gamma(G\diamond H)\leq\gamma(G)+\gamma(H)-1$ given
in Theorem \ref{Cor_basicUpper}, and notice that all pairs of graphs described
in Proposition \ref{one} and Theorem \ref{two}.(i) are sharp for this bound.
This also holds for many pairs that satisfies Theorem \ref{dom3}, but not for
all as can be seen from $P_{9}\diamond P_{9}$ for instance.

Beside that we present another example for sharpness of the bound
$\gamma(G\diamond H)\leq\gamma(G)+\gamma(H)-1$ that is not covered by above
mentioned results. Let $G$ be the graph on the left in Figure \ref{Pictire01},
and let $H_{k}=K_{2k}-M$ where $M$ is a perfect matching and $k\geq3$ ($H_{3}$
is the right graph of Figure \ref{Pictire01}). It is clear that $G$ and
$H_{k}$ satisfy the conditions given in Corollary \ref{biggertwo}, so
$\gamma(G\diamond H_{k})\geq3$. They do not satisfy (i), (ii), and (iii) in
Theorem \ref{dom3}, so let us consider the condition (iv) in Theorem
\ref{dom3}.

Since $H_{k}=K_{2k}-M$ where $M$ is a perfect matching, it follows that any
set $D_{H}=\{h_{1},h_{2},h_{3}\}\subseteq V(H_{k})$ which consists of
precisely three vertices satisfies $N_{H_{k}}[h_{1}]\cap N_{H_{k}}[h_{2}]\cap
N_{H_{k}}[h_{3}]\neq\emptyset,$ which means $a_{H}([3],D_{H})=1.$ The
condition (iv) of Theorem \ref{dom3} would then imply $a_{G}(\emptyset
,D_{G})=0$ for $D_{G}=\{g_{1},g_{2},g_{3}\},$ i.e. that $D_{G}$ is a
dominating set in $G.$ The only set of three vertices in $G$ which is a
dominating set consists of the black vertices shown in Figure \ref{Pictire01}.
For such a set $D_{G},$ it holds that $A_{G}(I,D_{G})\not =\emptyset$ for
every $I\subseteq\lbrack3]$ with $\left\vert I\right\vert =1.$ Also, for any
set $D_{H}=\{h_{1},h_{2},h_{3}\}\subseteq V(H_{k})$ a subset $I\subseteq
\lbrack3]$ with $\left\vert I\right\vert =1$ can be chosen so that
$A_{H}(I^{c},D_{H})\not =\emptyset.$ Hence, for such a set $I\subseteq
\lbrack3]$ we have $a_{G}(I,D_{G})+a_{H}(I,D_{H})=2,$ a contradiction. We
conclude that the condition (iv) is also not satisfied for $G\diamond H_{k},$
so Theorem \ref{dom3} implies $\gamma(G\diamond H_{k})\geq4$. Finally, since
$\gamma(G\diamond H_{k})\leq\gamma(G)+\gamma(H)-1=4$, we have the equality in
Theorem \ref{Cor_basicUpper}.

\begin{figure}[h]
\centering\includegraphics[scale=0.35]{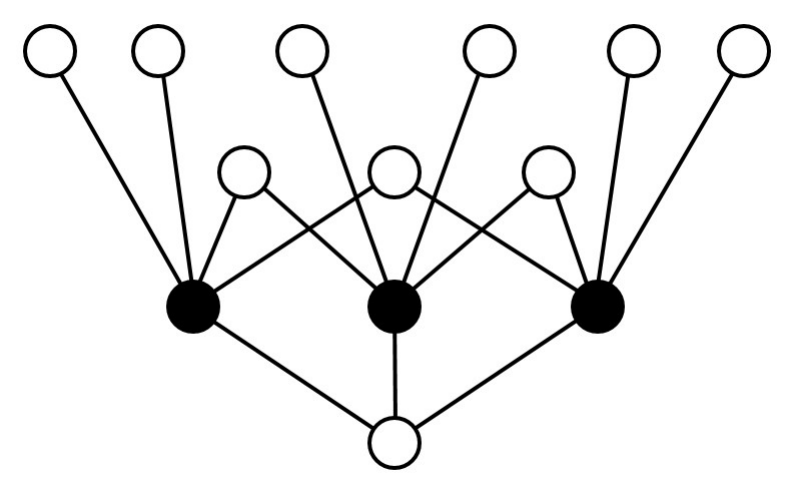}\hspace{1cm}%
\includegraphics[scale=0.35]{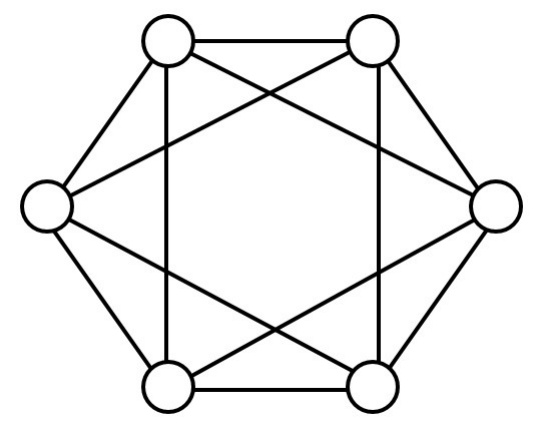}\caption{Graphs $G$ and $H$ such that
$\gamma(G\diamond H)=3+2-1=\gamma(G)+\gamma(H)-1$.}%
\label{Pictire01}%
\end{figure}

To finish this work, we show some conditions to have the domination number of
the modular product smaller than 6. We start with the condition on $G$ and $H$
under which $\gamma(G\diamond H)\leq4$.

\begin{proposition}
\label{4} Let $G$ be a graph such that $\mathrm{diam}(G)\geq3$, and let $H$ be
a graph which contains two adjacent vertices $h_{1},h_{2}$ such that
$N_{H}(h_{1})\cap N_{H}(h_{2})=\emptyset$. If $N_{H}(h_{1})\cup N_{H}%
(h_{2})\neq V(H)$, then $\gamma(G\diamond H)\leq4$.
\end{proposition}

\begin{proof}
Since $\mathrm{diam}(G)\geq3$, we can take two vertices $g_{1},g_{2}\in V(G)$
such that $d_{G}(g_{1},g_{2})=3$. We take $h_{3}\in V(H)\setminus(N_{H}%
(h_{1})\cup N_{H}(h_{2}))$ and let us see that
\[
D=\{(g_{1},h_{1}),(g_{1},h_{2}),(g_{2},h_{1}),(g_{2},h_{2}),(g_{1},h_{3})\}
\]
is a dominating set in $G\diamond H$. Let $(g,h)$ be any vertex in $G\diamond
H$. If $g\notin\{g_{1},g_{2}\}$ and $h\notin\{h_{1},h_{2}\}$, since
$N_{G}(g_{1})\cap N_{G}(g_{2})=\emptyset=N_{H}(h_{1})\cap N_{H}(h_{2})$, then
$(g,h)$ is dominated by $\{(g_{1},h_{1}),(g_{1},h_{2}),(g_{2},h_{1}%
),(g_{2},h_{2})\}$.

Next we suppose that $g\in\{g_{1},g_{2}\}$. If $(g,h)=(g_{1},h)$ (or
$(g,h)=(g_{2},h)$) and $h\notin N_{H}[h_{1}]$, then $(g,h)$ is adjacent to
$(g_{2},h_{1})$ (or $(g_{1},h_{1})$). If $(g,h)=(g_{1},h)$ (or $(g,h)=(g_{2}%
,h)$) and $h\in N_{H}(h_{1})$, then $(g,h)$ is adjacent to $(g_{1},h_{1})$ (or
$(g_{2},h_{1})$). Finally, let $h\in\{h_{1},h_{2}\}$. If $g\in N_{G}[g_{1}]$
(or $g\in N_{G}[g_{2}]$), then $(g,h)$ is equal or adjacent to $(g_{1},h)$ (or
to $(g_{2},h)$). If $g\notin N_{G}[g_{1}]\cup N_{G}[g_{2}]$, then $(g,h)$ is
adjacent to $(g_{1},h_{3})$. By Lemma \ref{red}, we get the result.
\end{proof}

\vspace{0.2cm}

There are many graphs satisfying the equality in the last proposition. We
present two families of products with Petersen graph as one factor. We use the
same notation as presented after Corollary \ref{examples}.

\begin{proposition}
If $k\geq10$, then
\[
\gamma(P_{k}\diamond P)=\gamma(C_{k}\diamond P)=4.
\]

\end{proposition}

\begin{proof}
By Proposition \ref{4} or Corollary \ref{peter}.(iv) we know that
$\gamma(P_{k}\diamond P)\leq4$. Corollary \ref{biggertwo} implies
$\gamma(P_{k}\diamond P)\geq3.$ We suppose to the contrary that $\gamma
(P_{k}\diamond P)=3$. Hence, $P_{k}$ and $P$ must satisfy at least one
condition of Theorem \ref{dom3}. Since $\gamma(P_{k})\geq4$ and $\gamma(P)=3,$
the condition (i) is not satisfied. Further, $\gamma(P_{k})\geq4$ implies
$\bar{\gamma}(P_{k})\geq4.$ As for the Petersen graph, every $\gamma(P)$-set
$D$ of $P$ consists of the three neighbors of a vertex $v\in V(P).$ This
implies that $v$ is not dominated by $D$ in $\overline{P},$ so $\bar{\gamma
}(P)\geq4.$ We conclude that the condition (ii) is not satisfied either. Since
$\mathrm{diam}(P)=2,$ the condition (iii) is not satisfied. We conclude that
the condition (iv) of Theorem \ref{dom3} must hold.

Let $D_{P_{k}}=\{g_{1},g_{2},g_{3}\}\subseteq V(P_{k})$ and $D_{P}%
=\{h_{1},h_{2},h_{3}\}\subseteq V(P)$ such that $a_{P_{k}}(I,D_{P_{k}}%
)+a_{P}(I^{c},D_{P})\leq1$ for every $I\subseteq\lbrack3].$ If $a_{P}%
([3],D_{P})=1,$ then $a_{P_{k}}(\emptyset,D_{P_{k}})=0,$ which implies that
$D_{P_{k}}$ is a dominating set in $P_{k},$ a contradiction with $k\geq10$.
Hence, it must hold that $a_{P}([3],D_{P})=0$. This implies that $D_{P}$ is
not a dominating set in $P,$ as every dominating set of $P$ consists of a
three neighbors of a same vertex $v\in V(P),$ so it would hold $v\in
A_{P}([3],D_{P}).$

Now, since $P_{k}$ is a path on $k\geq10$ vertices, it follows that there
exists at least one $I\subseteq\lbrack3]$ with $\left\vert I\right\vert =1,$
such that $a_{P_{k}}(I,D_{P_{k}})=1.$ This implies that $a_{P}(I^{c}%
,D_{P})=0.$ It is easily verified that every $D_{P}$ which is not a dominating
set in $G$ and satisfies $a_{P}(I^{c},D_{P})=0$ induces a path of the length
two in $P.$ Let $h$ be the internal vertex of such an induced path of the
length two in $P.$ Then $h\in N_{P}[h_{i}]$ for every $i\in\lbrack3],$ which
implies $a_{P}([3],D_{P})=1,$ a contradiction.

Hence no condition of Theorem \ref{dom3} is fulfilled, which means that
$\gamma(P_{k}\diamond P)\geq4$ and the equality follows. The same proof works
for $C_{k}\diamond P$.
\end{proof}

After we have given the example of graphs attaining the upper bound of
Proposition \ref{4}, we present a condition on graphs $G$ and $H$ under which
$\gamma(G\diamond H)\leq5$.

\begin{proposition}
\label{Prop_Izrok}Let $G$ and $H$ be two graphs such that they contains two
adjacent vertices $g_{1},g_{2}$ and $h_{1},h_{2}$, respectively, such that
$N_{G}(g_{1})\cap N_{G}(g_{2})=\emptyset$ and $N_{H}(h_{1})\cap N_{H}%
(h_{2})=\emptyset$. If $N_{G}(g_{1})\cup N_{G}(g_{2})\neq V(G)$ and
$N_{H}(h_{1})\cup N_{H}(h_{2})\neq V(H)$, then $\gamma(G\diamond H)\leq5$.
\end{proposition}

\begin{proof}
We take $g_{3}\in V(G)\setminus(N_{G}(g_{1})\cup N_{G}(g_{2}))$ and $h_{3}\in
V(H)\setminus(N_{H}(h_{1})\cup N_{H}(h_{2}))$ and let us see that
\[
D^{\prime}=\{(g_{1},h_{1}),(g_{1},h_{2}),(g_{2},h_{1}),(g_{2},h_{2}%
),(g_{2},h_{3}),(g_{3},h_{1})\}
\]
is a dominating set in $G\diamond H$. Let $(g,h)$ any vertex in $G\diamond H$.
Firstly, if $g\notin\{g_{1},g_{2}\}$ and $h\notin\{h_{1},h_{2}\}$, since
$N_{G}(g_{1})\cap N_{G}(g_{2})=\emptyset=N_{H}(h_{1})\cap N_{H}(h_{2})$, then
$(g,h)$ is dominated by $\{(g_{1},h_{1}),(g_{1},h_{2}),(g_{2},h_{1}%
),(g_{2},h_{2})\}$. If $(g,h)=(g_{1},h)$ (or $(g,h)=(g_{2},h)$) and $h\notin
N_{H}(h_{1})$, then $(g,h)$ is adjacent to $(g_{3},h_{1})$. If $h\in
N_{H}(h_{1})$, then $(g,h)$ is adjacent to $(g_{1},h_{1})$ (or $(g_{2},h_{1}%
)$). Finally, if $(g,h)=(g,h_{1})$ and $g\in N_{G}(g_{1})$ (or $g\in
N_{G}(g_{2})$), then $(g,h_{1})$ is adjacent to $(g_{1},h_{1})$ (or
$(g_{2},h_{1})$). If $g\notin N_{G}(g_{1})\cup N_{G}(g_{2})$, then $(g,h_{1})$
(or $(g,h_{2})$) is adjacent to $(g_{2},h_{3})$. Now, by Lemma \ref{red}, we
get the result.
\end{proof}

\section{Conclusions and further work}

In this work we study the domination number of the modular product $G\diamond
H$ of two graphs $G$ and $H.$ First, several upper and lower bounds for
$\gamma(G\diamond H)$ in terms of $\gamma(G),$ $\gamma(H)$, $\gamma
_{t}(\overline{G})$ and $\gamma_{t}(\overline{H})$ are established. These
results are useful for characterizing graphs $G$ and $H$ with $\gamma
(G\diamond H)=k$ for every $k\in[3]$ and we completely characterize such
graphs. Several classes of graphs are provided for which the domination number
of the modular product is bounded, and also two classes of graphs for which it
is unbounded. In the last section, we give several conditions on $G$ and $H$
under which $\gamma(G\diamond H)$ is at most $4$ and $5.$ For the upper bound
$\gamma(G\diamond H)\leq5$ of {Proposition \ref{Prop_Izrok}}, we do not have a
pair of graphs attaining it, so we leave it as an open problem.

\begin{problem}
Find{ a pair of graphs $G$ and $H$ satisfying conditions of Proposition
\ref{Prop_Izrok} for which }$\gamma(G\diamond H)=5.$
\end{problem}

Another question of our interest is to determine the magnitude of
$\gamma(G\diamond H)$ when the diameter of both $G$ and $H$ equals two. Recall
the upper bound $\gamma(G\diamond H)\leq\gamma(G)+\gamma(H)-1$ from Theorem
\ref{Cor_basicUpper}. Corollary \ref{Kn} implies that the upper bound reduces
to $\gamma(G\diamond H)\leq\gamma(G)$ if $H$ is of diameter $1$ (the complete
graph), Corollary \ref{Cor_13} implies that it reduces to $\gamma(G\diamond
H)\leq\gamma(G)+2$ if $G$ is of diameter at least $3$, and Corollary
\ref{diam5} reduces to $\gamma(G\diamond H)\leq\gamma(G)$ if $G$ is of
diameter at least $5$. Even more, if both graphs are of diameter at least $3,$
it again reduces to $\gamma(G\diamond H)\leq3$ according to Proposition
\ref{diameter3}. Hence, the question is whether the bound can be similarly
reduced in the case when both $G$ and $H$ are of diameter $2.$ In some
specific situations, such as when $H$ has a universal vertex the bound is
reduced to $\gamma(G\diamond H)\leq\gamma(G)$ according to Corollary
\ref{univ}. However, for the general case when both $G$ and $H$ are graphs of
diameter $2$ we only have the bound of Theorem \ref{Cor_basicUpper}.

\begin{problem}
Find the magnitude of $\gamma(G\diamond H)$ in terms of $\gamma(G)$ and
$\gamma(H)$ when these two graphs are of diameter $2.$
\end{problem}

To simplify, one can consider a special case of the above problem when $G=H$.
Notice that by Theorem \ref{Cor_basicUpper} we have
\[
\min\{\gamma(G),\gamma_{t}(\overline{G})\}\leq\gamma(G\diamond G)\leq
\min\{2\gamma(G)-1,2\gamma_{t}(\overline{G})-1\}.
\]
We have considered this special case and all graphs we have are with the
property $\gamma(G\diamond G)\leq\gamma(G)+1.$ Mind, this can be misleading as
we only considered small graphs. Hence, we are inclined to propose the
following problem.

\begin{problem}
Find{ graphs $G$ satisfying }$\gamma(G\diamond G)\geq\gamma(G)+2.$
\end{problem}

\noindent By the results of this paper, observe that if such a graph exists it
must be of the diameter 2.

As an useful notion in this paper we introduced a simultaneously dominating
and complement total dominating set (SDCTD set for short) of $G$ defined as
any set $D$ which is at a same time a dominating set of $G$ and a total
dominating set of $\overline{G}.$ The minimum cardinality of SDCTD set of $G$
is denoted by $\bar{\gamma}(G)$ and is called the SDCTD number of $G$. We
believe this invariant could be of an independent interest. The ILP model for
determining a smallest SDCTD set $D$ of $G$ is given by%
\[%
\begin{tabular}
[c]{ll}
& $\min\sum\limits_{v\in V(G)}x_{v}\medskip$\\
s.t. & $\sum\limits_{w\in N_{G}(v)}x_{w}+x_{v}\geq1$ for any $v\in V(G),$\\
& $\sum\limits_{w\in N_{\overline{G}}(v)}x_{w}\geq1$ for any $v\in V(G),$\\
& $x_{v}\in\{0,1\}$ for any $v\in V(G),$%
\end{tabular}
\ \ \ \ \
\]
and then $D=\{v\in V(G):x_{v}=1\}$. The behaviour of this invariant to other
graph products such as Cartesian product would be interesting to research.

\bigskip

\bigskip\noindent\textbf{Acknowledgments.}~~S. B. was partly supported by Plan
Nacional I+D+I Grant PID2021-127842NB-I00 and PID2022-139543OB-C41. I. P. was
partially supported by the Slovenian Research Agency program No. P1-0297. J.
S. and R. \v{S}. have been partially supported by the Slovenian Research
Agency ARRS program P1-0383 and ARRS project J1-3002, J. S. also acknowledges
the support of Project KK.01.1.1.02.0027, a project co-financed by the
Croatian Government and the European Union through the European Regional
Development Fund - the Competitiveness and Cohesion Operational Program.

\bigskip\noindent\textbf{Declaration of interests.}~~The authors declare that
they have no known competing financial interests or personal relationships
that could have appeared to influence the work reported in this paper.

\bigskip\noindent\textbf{Data availability.}~~Our manuscript has no associated data.

\end{document}